\documentclass[letterpaper, 10 pt, conference]{ieeeconf}  

\IEEEoverridecommandlockouts                              
\usepackage{graphicx} 
\usepackage{times} 

\usepackage{amsfonts,amsmath,amssymb} 

\usepackage{array}
\usepackage{subfigure}

\graphicspath{{Figures/}}

\usepackage[utf8]{inputenc}
\usepackage[T1]{fontenc}
\usepackage[english]{babel} 
\usepackage{cite}

\usepackage{subfigure}

\usepackage{amsmath}
\usepackage{amssymb}
\usepackage{eucal}

\newtheorem{thm}{Theorem}
\newtheorem{lem}{Lemma}

\providecommand{\abs}[1]{\left\lvert#1\right\rvert}

\newcommand{\R}{\mathbb R}

\newcommand{\N}{\mathbb N}

\newcommand{\be}{\begin{equation}}
\newcommand{\ee}{\end{equation}}

\newcommand{\bes}{\begin{equation*}}
\newcommand{\ees}{\end{equation*}}

\newcommand{\ba}{\begin{eqnarray}}
\newcommand{\ea}{\end{eqnarray}}

\newcommand{\ep}{\varepsilon}

\title{\LARGE \bf Null Controllability of the 2D Heat Equation Using Flatness}

\author{Philippe Martin, Lionel Rosier and Pierre Rouchon
\thanks{P.~Martin and P.~Rouchon are with the Centre Automatique et Systèmes, MINES ParisTech, 75006~Paris,~France
{\tt\footnotesize \{philippe.martin, pierre.rouchon\}@mines-paristech.fr}}%
\thanks{L.~Rosier is with Institut Élie Cartan, UMR 7502 UdL/CNRS/INRIA, BP 70239, 54506 Vandœuvre-lès-Nancy, France
{\tt\footnotesize Lionel.Rosier@univ-lorraine.fr}}%
}

\begin{document}

\maketitle
\thispagestyle{empty}
\pagestyle{empty}

\begin{abstract}
We derive in a direct and rather straightforward way the null controllability of a 2-D heat equation with boundary control.
We use the so-called {\em flatness approach}, which consists in parameterizing the solution and the control by the derivatives of a ``flat output''. This provides an explicit control law achieving the exact steering to zero. Numerical experiments  demonstrate the relevance of the approach.
\end{abstract}



\section{Introduction}
%
The controllability of the heat equation was first considered in the 1-D case, \cite{FattoR1971ARMA,Jones1977JMAA,Littm1978ASNSPCS}), and very precise results were obtained by
the classical moment approach.  Next using Carleman estimates and duality arguments the null controllability was proved in \cite{FursiI1996book,LebeaR1995CPDE} for any bounded domain in~$\R^N$, any control time $T$, and any control region. This Carleman approach proves very efficient also with semilinear parabolic equations, \cite{FursiI1996book}.

By contrast the numerical control of the heat equation (or of parabolic equations) is in its early stage, see e.g.~\cite{MunchZ2010IP,BoyerHL2011NM,MicuZ2011SCL}. A natural candidate for the control input is the control of minimal $L^2-$norm, which may be obtained as a trace of the solution of the (backward) adjoint problem whose terminal state is the minimizer of a suitable quadratic cost. Unfortunately its computation is a hard task~\cite{MicuZ2011SCL}; indeed the terminal state of the adjoint problem associated with some regular initial state of the control problem may be highly irregular, which leads to severe troubles in the numerical computation of the control function.

All the above results rely on some observability inequalities for the adjoint system. A direct approach which does not involve the adjoint problem was proposed in~\cite{Jones1977JMAA,Littm1978ASNSPCS,LinL1995AMO}.  In \cite{Jones1977JMAA} a fundamental solution for the heat equation with compact support in time was introduced and used to prove null controllability. The results in \cite{Jones1977JMAA,Rosie2002CAM} can be used to derive control results on a bounded interval with one boundary control in some Gevrey class. An extension of those results to the semilinear heat equation in 1D was obtained in \cite{LinL1995AMO} in a more explicit way through the resolution of an ill-posed problem with data of Gevrey order 2 in~$t$.

In this paper we derive in a straightforward way the null controllability of the 2-D heat equation
\begin{IEEEeqnarray}{rCl}
    \theta_t(t,x_1,x_2) &=&\Delta\theta(t,x_1,x_2)\label{eq:heat}\\
    \theta_{x_1}(t,0,x_2) &=&0\label{eq:bc1}\\
    \theta_{x_1}(t,L,x_2) &=&0\label{eq:bc2}\\
    \theta_{x_2}(t,x_1,0) &=&0\label{eq:bc3}\\
    \theta_{x_2}(t,x_1,1) &=&u(t,x_1)\label{eq:bcu}
\end{IEEEeqnarray}
with initial condition
\begin{IEEEeqnarray*}{rCl}
    \theta(0,x_1,x_2) &=&\theta_0(x_1,x_2);\label{eq:ic}
\end{IEEEeqnarray*}
here $(x_1,x_2)$ belongs to the rectangle $\Omega:=(0,L)\times(0,1)$, $t\in(0,T)$, and $\theta_0\in L^2(\Omega)$.
This system describes the dynamics of the temperature~$\theta$ in an insulated metal strip where the control~$u$
is the heat flux at one end.

More precisely given any final time~$T>0$ and any initial state $\theta_0\in L^2(\Omega)$
we provide an explicit very regular control input~$u$ such that the state reached at time~$T$ is zero, i.e.
\begin{IEEEeqnarray*}{rCl'l}
    \theta(T,x_1,x_2) &=&0, &x\in\Omega.
\end{IEEEeqnarray*}
We use the so-called {\em flatness approach}, \cite{FliesLMR1995IJoC}, which consists in parameterizing the solution~$\theta$ and the control~$u$ by the derivatives of a ``flat output''~$y$ (section~\ref{sec:flatness}); this notion was initially introduced for finite-dimensional (nonlinear) systems, and later extended to (in particular) parabolic PDEs, \cite{LarocMR2000IJRNC,LynchR2002IJC,MeureZ2008MCMDS,Meure2011A}. Choosing a suitable trajectory for this flat output~$y$ then yields an explicit series for a control achieving the exact steering to zero (section~\ref{sec:controllability}). Numerical experiments  demonstrate the relevance of the approach (section~\ref{sec:numerics}).

This paper extends~\cite{MartiRR2013arxiv} to the 2-D case, with moreover a more elegant and efficient construction of the control. This is probably the first example using flatness for motion planning of a ``truly'' 2-D PDE.

In the sequel we will consider series with infinitely many derivatives of some functions. The notion of Gevrey order is a way of estimating the growth of these derivatives: we say that a function $y\in C^\infty([0,T])$ is {\em Gevrey of order $s\geq0$ on~$[0,T]$} if there exist positive constants~$M,R$ such that
\bes
\abs{y^{(p)}(t)} \leq M\frac{p!^s}{R^p} \qquad \forall t\in [0,T],\ \forall p\ge 0.
\ees
More generally if $K\subset\R^N$ is a compact set and $y$ is a function of class~$C^\infty$ on~$K$ (i.e. $y$ is the restriction to $K$ of a function of class~$C^\infty$ on some open neighbourhood $\Omega$ of~$K$), we say $y$ is {\em Gevrey of order $s_1$ in $x_1$, $s_2$ in $x_2$,\ldots,$s_N$ in $x_N$ on~$K$} if there exist positive constants $M,R_1,...,R_N$ such that
for all $x\in K$ and $p\in\N^N$
\bes
\abs{\partial_{x_1}^{p_1}\partial_{x_2}^{p_2}\cdots\partial_{x_N}^{p_N}y(x)} \le
M\frac{\prod_{i=1}^N (p_i !)^{s_i}}{\prod_{i=1}^N R_i^{p_i} }.
\ees
By definition, a Gevrey function of order $s$ is also of order $r$ for~$r\geq s$. Gevrey functions of order~1 are analytic (entire if $s<1$).  Gevrey functions of order~$s>1$ have a divergent Taylor expansion; the larger~$s$, the ``more divergent'' the Taylor expansion. Important properties of analytic functions generalize to Gevrey functions of order $s>1$: the scaling, addition, multiplication and derivation of Gevrey functions of order $s>1$ is of order~$s$, see~\cite{Ramis1978,Rudin1987book}.  But contrary to analytic functions, functions of order $s>1$ may be constant on an open set without being constant everywhere.  For example the ``step function''
\bes
\phi_s(t):=\begin{cases}
1 & \text{if $t\leq0$}\\
0 & \text{if $t\geq1$}\\
\dfrac{ e^{-(1-t)^{-k}} }{ e^{-(1-t)^{-k}} + e^{-t^{-k}} }
&\text{if $t\in]0,1[$},
\end{cases}
\ees
where $k = (s-1)^{-1}$, is Gevrey of order~$s$ on $[0,1]$ (and in fact on~$\R$); notice $\phi_s(0)=1$, $\phi_s(1)=0$ and $\phi_s^{(i)}(0)=\phi_s^{(i)}(1)=0$ for all~$i\geq1$.

In conjunction with growth estimates we will repeatedly use Stirling's formula $n!\sim(n/e)^n\sqrt{2\pi n}$.

\section{The heat equation is ``flat''}\label{sec:flatness}
Let $(e_j)_{j\ge 0}$ be an orthonormal basis of~$L^2(0,L)$ such that each function $e_j$ is an eigenfunction for the Neumann Laplacian on $(0,L)$, ie is a solution of
\begin{IEEEeqnarray*}{rCl'l}
    -e_j''(x_1) &=&\lambda _j e_j(x_1), &x_1\in(0,L)\\
    e_j'(0) &=&0\\
    e_j'(L) &=&0.
\end{IEEEeqnarray*}
For instance $\lambda_j=\Bigl(\frac{j\pi}{L}\Bigr)^2$ and
\begin{IEEEeqnarray*}{rCl'l}
    e_0(x_1) &=& L^{-\frac{1}{2}}\\
    e_j(x_1) &=&\sqrt\frac{2}{L}\cos\Bigl(\frac{j\pi x_1}{L}\Bigr), &j\geq1.
\end{IEEEeqnarray*}
Decompose then $\theta(t,x_1,0)$ as
\begin{IEEEeqnarray*}{rCl}
    \theta(t,x_1,0) &=&\sum_{j\geq0}z_j(t)e_j(x_1).
\end{IEEEeqnarray*}
We claim~\eqref{eq:heat}--\eqref{eq:bcu} is ``flat'' with $z(t):=\bigl(z_j(t)\bigr)_{j\geq0}$ as a flat output, which means there is (in appropriate spaces of smooth functions) a $1-1$ correspondence between arbitrary functions $t\mapsto z(t)$ and solutions of~\eqref{eq:heat}--\eqref{eq:bcu}.

On the one hand a solution~$\theta$ of~\eqref{eq:heat}--\eqref{eq:bcu} uniquely defines the $z_k$'s hence $z$ by
\begin{IEEEeqnarray*}{rCl}
    \int_0^L\theta(t,x_1,0)e_k(x_1)dx_1 &=&\sum_{j\geq0}z_j(t)\int_0^Le_j(x_1)e_k(x_1)dx_1\\
    &=& z_k(t).
\end{IEEEeqnarray*}
On the other hand let us seek a formal solution in the form
\begin{IEEEeqnarray*}{rCl}
    \theta(t,x_1,x_2) &=&\sum_{i\geq0}\frac{x_2^i}{i!}a_i(t,x_1),
\end{IEEEeqnarray*}
where the $a_i$'s are functions yet to define. Plugging this formal solution into~\eqref{eq:heat} we find
\begin{IEEEeqnarray*}{rCl}
    \sum_{i\geq0}\frac{x_2^i}{i!}\bigl[a_{i+2}(t,x_1)-(\partial_t-\partial^2_{x_1})a_i(t,x_1)\bigr] &=&0,
\end{IEEEeqnarray*}
hence $a_{i+2}=(\partial_t-\partial^2_{x_1})a_i$ for all~$i\geq0$. Moreover $a_0(t,x_1)=\theta(t,x_1,0)$ and by~\eqref{eq:bc3} $a_1(t,x_1)=0$. Therefore for all~$i\geq0$
\begin{IEEEeqnarray*}{rCl}
    a_{2i+1}(t,x_1) &=& 0\\
    a_{2i}(t,x_1) &=& (\partial_t-\partial^2_{x_1})^ia_0\\
    &=& \sum_{j\geq0}(\partial_t-\partial^2_{x_1})^iz_j(t)e_j(x_1)\\
    &=& \sum_{j\geq0}e_j(x_1)(\partial_t+\lambda_j)^iz_j(t)\\
    &=& \sum_{j\geq0}e_j(x_1)e^{-\lambda_jt}y_j^{(i)}(t),
\end{IEEEeqnarray*}
where we have set $y_j(t):= e^{\lambda _j t} z_j(t)$. Clearly
\begin{IEEEeqnarray}{rCl}
    \theta(t,x_1,x_2) &=&\sum_{j\geq0}e^{-\lambda_jt}e_j(x_1)
    \sum_{i\ge0}y_j^{(i)}(t)\frac{x_2^{2i}}{(2i)!}\label{eq:theta}\\
    u(t,x_1) &=& \sum_{j\geq0}e^{-\lambda_jt}e_j(x_1)
    \sum_{i\ge0}\frac{1}{(2i-1)!}y_j^{(i)}(t). \label{eq:u}
\end{IEEEeqnarray}
is then a formal solution of~\eqref{eq:heat}--\eqref{eq:bcu} uniquely defined by the~$y_j$'s hence by~$z$.

We now give a precise meaning to this formal solution by restricting the~$y_j$'s to be Gevrey of order~$s\in [1,2)$.
\begin{thm}\label{prop:regularity}
Let $s\in [1,2)$, $ 0 < t_1<t_2<\infty$, and a sequence $y=(y_j)_{j\ge 0}$ in $C^\infty ([0,T])$
satisfying for some constants $M,R >0$
\begin{IEEEeqnarray*}{rCl'l}
\abs{y_j^{(i)}(t)} &\leq& M\frac{i!^s}{R^i},  &\forall i,j\ge 0,\ \forall t\in[t_1,t_2].
\end{IEEEeqnarray*}
Then the formal solution~$\theta$ defined by~\eqref{eq:theta} is Gevrey of order~$s$ in~$t$, $1/2$ in~$x_1$, and~$s/2$ in~$x_2$ on~$[t_1,t_2]\times\overline{\Omega}$; as a consequence the formal control~$u$ defined by~\eqref{eq:u} is also Gevrey of order~$s$ in~$t$ and $1/2$ in~$x_1$ on~$[t_1,t_2]\times[0,L]$.
\end{thm}
\begin{proof}
We must prove the formal series
\begin{IEEEeqnarray*}{rCl}
    \partial^m_t\partial^{2l}_{x_1}\partial^n_{x_2}\theta(t,x_1,x_2)
    &=&
    \sum_{2i\geq n}\sum_{j\geq0}E_{i,j},
\end{IEEEeqnarray*}
where
\begin{IEEEeqnarray*}{rCl}
    E_{i,j} &:=& \frac{x_2^{2i-n}}{(2i-n)!}(-\lambda_j)^le_j(x_1)
    \partial_t^m\bigl(e^{-\lambda_jt}y_j^{(i)}(t) \bigr),
\end{IEEEeqnarray*}
are uniformly convergent  on $[t_1,t_2]\times \overline{\Omega}$ with growth estimates of the form
\begin{IEEEeqnarray*}{rCl}
    \abs{\partial^m_t\partial^{2l}_{x_1}\partial^n_{x_2}\theta(t,x_1,x_2)}
    &\leq&
    C\frac{m!^s}{R_0^m}\,\frac{(2l)!^\frac{1}{2}}{R_1^{2l}}\,\frac{n!^\frac{s}{2}}{R_2^n}
\end{IEEEeqnarray*}
for some positive constants $R_0,R_1,R_2$. In what follows $C$ denotes a generic constant independent of $m,l,n,t,x_1,x_2$ that may vary from line to line. By assumption
\begin{IEEEeqnarray*}{rCl}
    \abs{E_{i,j}} &=& \abs{ \frac{x_2^{2i-n}}{(2i-n)!}e_j(x_1)
    \sum_{k=0}^m\binom{m}{k} y_j^{(i+k)}(t) \lambda_j^{l+m-k}e^{-\lambda_jt} }\\
    &\le& \frac{M\sqrt\frac{2}{L}}{(2i-n)!}
    \sum_{k=0}^m \frac{m!}{k! (m-k) !}\frac{(i+k)!^s}{R^{i+k}} \lambda_j^{l+m-k}e^{-\lambda_jt_1}.
\end{IEEEeqnarray*}

Let $t_0\in(0,t_1)$ and $\delta:=t_1-t_0$. Using $(p+q)! \le 2^{p+q}p!q!$ and $x^pe^{-x}\le p!$ for $x\ge0$
\begin{IEEEeqnarray*}{rClCl}
    \frac{(i+k)!^s}{k!R^{i+k}}
    &\leq& \frac{i!^sk!^s2^{s(i+k})}{k!R^{i+k}}
    &\leq& i!^sm!^{s-1}\Bigl(\frac{2^s}{R}\Bigr)^{i+k}\\
    \lambda_j^{l+m-k}e^{-\lambda_j\delta}
    &\leq& \frac{(l+m-k)!}{t_1^{l+m-k}}
    &\leq& l!(m-k)!\Bigl(\frac{2}{\delta}\Bigr)^{l+m-k},
\end{IEEEeqnarray*}
so we have
\begin{IEEEeqnarray*}{rCl}
    \abs{E_{i,j}} &\le& e^{-\lambda_jt_0}M\sqrt\frac{2}{L}\, \frac{i!^s\Bigl(\frac{2^s}{R}\Bigr)^i}{(2i-n)!}\,
    l!m!^s\Bigl(\frac{2}{\delta}\Bigr)^{l+m} \sum_{k=0}^m \Bigl(\frac{2^s\delta}{2R}\Bigr)^k.
\end{IEEEeqnarray*}

Now by Stirling's formula
\begin{IEEEeqnarray*}{rClCl}
    l!\Bigl(\frac{2}{t_1}\Bigr)^l &\sim& (2l)!^\frac{1}{2}\frac{(\pi l)^\frac{1}{4}}{\delta^l}
    &\leq& C \frac{(2l)!^\frac{1}{2}}{R_1^{2l}},
\end{IEEEeqnarray*}
where $R_1<\sqrt{\delta}$. Likewise using also $(p+q)! \le 2^{p+q}p!q!$
\begin{IEEEeqnarray*}{rCl}
    \frac{i!^s\Bigl(\frac{2^s}{R}\Bigr)^i}{(2i-n)!}
    &\sim& \Bigl(\frac{2^s}{R}\Bigr)^i \frac{(2i)!^\frac{1}{2}}{(2i-n)^\frac{1}{2}!}
    \frac{(\pi i)^\frac{s}{4}}{(2i-n)!^{1-\frac{s}{2}}2^{is}}\\
    &\leq& \Bigl(\frac{2^s}{R}\Bigr)^i (\pi i)^\frac{s}{4} \frac{n!^\frac{s}{2}}{(2i-n)!^{1-\frac{s}{2}}}\\
    &\leq& \Bigl(\frac{2^s}{R}\Bigr)^i \Bigl(\frac{n\pi}{2}\Bigr)^\frac{s}{4}
    \biggl(1+\Bigl(\frac{2i-n}{n}\Bigr)^\frac{s}{4}\biggr) \frac{n!^\frac{s}{2}}{(2i-n)!^{1-\frac{s}{2}}}\\
    &\leq& C \frac{n!^\frac{s}{2}}{R_2^n}
    \underbrace{ \frac{1+\Bigl(\frac{2i-n}{n}\Bigr)^\frac{s}{4}}{(2i-n)!^{1-\frac{s}{2}}}
    \sqrt{\frac{2^s}{R}}^{2i-n} }_{A_{2i-n}},
\end{IEEEeqnarray*}
where $R_2<\sqrt{\frac{R}{2^s}}$; notice $\sum_{2i\geq n}{A_{2i-n}}<\infty$. Finally notice
\begin{IEEEeqnarray*}{rCl}
    \Bigl(\frac{2}{\delta}\Bigr)^m \sum_{k=0}^m \Bigl(\frac{2^s\delta}{2R}\Bigr)^k
    &\leq& C\frac{1}{R_0^m}
\end{IEEEeqnarray*}
if we set
\begin{IEEEeqnarray*}{rCl}
    R_0 &:=& \begin{cases}
    \frac{R}{2^s}& \text{if $\frac{2^s\delta}{2R}>1$},\\
    \frac{\delta}{4}& \text{if $\frac{2^s\delta}{2R}=1$},\\
    \frac{\delta}{2}& \text{if $\frac{2^s\delta}{2R}<1$}.
\end{cases}
\end{IEEEeqnarray*}

Collecting the three previous estimates yield
\begin{IEEEeqnarray*}{rCl}
    \sum_{2i\geq n}\sum_{j\geq0}E_{i,j} &\leq&
    C \frac{m!^s}{R_0^m}\,\frac{(2l)!^\frac{1}{2}}{R_1^{2l}}\,\frac{n!^\frac{s}{2}}{R_2^n}
    \sum_{j\geq0}e^{-\lambda_jt_0} \sum_{2i\geq n}{A_{2i-n}}.
\end{IEEEeqnarray*}
As the two series in the rhs converge the claim is proved.
\end{proof}

\section{Null controllability}\label{sec:controllability}
In this section we derive an explicit control steering the system from any initial state~$\theta_0\in L^2(\Omega)$ at time~$0$ to the final state~$0$ at time~$T>0$. Two ideas are involved: on the one hand thanks to the flatness property it is easy, by selecting a suitable trajectory for the flat output, to find a control achieving the steering to zero starting from a certain set of initial conditions; on the other hand thanks to the regularizing property of the heat equation this set is reached from any~$\theta_0\in L^2(\Omega)$ when applying first a zero control for some time (lemma~\ref{lem:zerocontrol}).
\begin{lem}\label{lem:zerocontrol}
Let $\theta_0\in L^2(\Omega)$ and $\tau>0$. Consider the final state $\theta_\tau(x_1,x_2):=\theta(\tau,x_1,x_2)$ reached when applying the control $u(t):=0$, $t\in[0,\tau]$, starting from the initial state~$\theta_0$.

Then $\theta_\tau$ can be expanded as
\begin{IEEEeqnarray*}{rCl}
\theta_\tau(x_1,x_2) &=& \sum_{j\ge0}e^{-\lambda_j\tau}e_j(x_1)\sum_{i\ge0}y_{j,i}\frac{x_2^{2i}}{(2i)!},
\end{IEEEeqnarray*}
with
\begin{IEEEeqnarray*}{rCl}
\abs{y_{j,i}} &\le& C\Bigl(1+\frac{1}{\sqrt\tau}\Bigr)\frac{i!}{\tau^i}
\end{IEEEeqnarray*}
where $C$ is some positive constant depending only on~$\theta_0$.
\end{lem}
\begin{proof}
We can decompose $\theta _0$ as the Fourier series
\bes
\theta_0(x_1,x_2)=\sum_{j,n\ge0}c_{j,n}e_j(x_1)\sqrt{2}\cos(n\pi x_2)
\ees
where the convergence holds in $L^2(0,1)$ and
\bes
\sum_{j,n\ge0}\abs{c_{j,n}}^2 <\infty.
\ees
The solution of~\eqref{eq:heat}-\eqref{eq:bcu} starting from~$\theta_0$ then reads
\bes
\theta(t,x_1,x_2)=\sum_{j,n\ge0}c_{j,n}e^{-(\lambda_j+n^2\pi^2)t}e_j(x_1)\sqrt{2}\cos(n\pi x_2)
\ees
and in particular
\bes
\theta_\tau(x_1,x_2)=\sum_{j,n\ge0}c_{j,n}e^{-(\lambda_j+n^2\pi^2)\tau}e_j(x_1)\cos(n\pi x_2).
\ees
This series obviously uniformly converges since
\bes
\abs{c_{j,n}e^{-(\lambda_j+n^2\pi^2)\tau}e_j(x_1)\sqrt{2}\cos(n\pi x_2)}
\leq C_1e^{-n^2\pi^2\tau}. 
\ees
Moreover
\begin{IEEEeqnarray*}{rCl}
\theta_\tau(x_1,x_2) &=& \sqrt{2}\sum_{j,n\ge0}c_{j,n}e^{-(\lambda_j+n^2\pi^2)\tau}e_j(x_1)\\
&&\quad\times\sum_{i\ge0}(-1)^i\frac{(n\pi x_2)^{2i}}{(2i)!}\\
&=& \sum_{j\ge0}e^{-\lambda_j\tau}e_j(x_1)\sum_{i\ge0}\frac{x_2^{2i}}{(2i)!}\\
&&\quad\times\underbrace{\sqrt{2}(-1)^i\sum_{n\ge0}c_{j,n}e^{-n^2\pi^2\tau}\frac{(n\pi)^{2i}}{(2i)!}}_{=:y_{j,i}}.
\end{IEEEeqnarray*}

\begin{IEEEeqnarray*}{rCl}
\theta_\tau(x) &=& \sqrt{2}\sum_{n\ge 0}c_ne^{-n^2\pi^2\tau}\sum_{i\ge0}(-1)^i\frac{(n\pi x)^{2i}}{(2i)!}\\
%
%
&=& \sum_{i\ge0} \frac{x^{2i}}{(2i)!} \underbrace{\left(\sqrt{2}(-1)^i\sum_{n\ge0}c_n e^{-n^2\pi^2\tau}(n\pi)^{2i}\right)}_{=:y_i}
\end{IEEEeqnarray*}
The change in the order of summation will be justified once we have proved that $y_{j,i}$, $i\geq0$, is absolutely convergent and
\bes
\sum_{i\ge0}\abs{y_{j,i}}\frac{x_2^{2i}}{(2i)!} <\infty,  \quad\forall x_2\geq0.
\ees
For $i\ge 0$ let $h_i(x):=e^{-\tau\pi^2x^2}(\pi x)^{2i}$ and $N_i:=\left[\bigl(\frac{i}{\pi ^2\tau}\bigr)^{\frac{1}{2}}\right]$. The map $h_i$ is increasing on $\bigl[0,\left(\frac{i}{\pi ^2\tau}\right)^{\frac{1}{2}}\bigr]$ and decreasing on $\bigl[\left(\frac{i}{\pi ^2\tau}\right)^{\frac{1}{2}},+\infty\bigr)$ hence
\begin{IEEEeqnarray*}{rCl}
\sum_{n\ge0}h_i(n) &\le& \int_0^{N_i}h_i(x)dx + h_i(N_i)\\
&& +\: h_i (N_i + 1) + \int_{N_{i+1} }^\infty h_i(x)dx\\
&\le& 2h_i\left(\Bigl(\frac{i}{\pi^2\tau}\Bigr)^{\frac{1}{2}}\right) + \int_0^\infty h_i(x)dx\\
&\le& C_2\frac{i!}{\tau^i\sqrt{i}} + \int_0^\infty h_i(x)dx;
\end{IEEEeqnarray*}
$C_2$ is some positive constant and we have used Stirling's formula. On the other hand integrating by parts yields
\begin{IEEEeqnarray*}{rCl}
\int_0^\infty h_i(x)dx &=& \frac{2i-1}{2\tau}\int_0^\infty h_{i-1}(x)dx\\
&=& \frac{(2i-1)\cdots 3\cdot1}{(2\tau)^i}\int_0^\infty e^{-\tau\pi^2x^2}dx\\
&=& \frac{(2i)!}{2^ii!(2\tau)^i}\cdot\frac{1}{\pi\sqrt\tau} \int_0^\infty e^{-x^2}dx\\
&\leq& C_3\frac{i!}{\tau^i\sqrt{i\tau}},
\end{IEEEeqnarray*}
where $C_3$ is some positive constant and we have again used Stirling's formula. As a consequence
\bes
\abs{y_{j,i}}\le\sqrt{2}\sup_{j,n\ge 0}\abs{c_{j,n}}\sum_{n\ge0}h_i(n) \le C\Bigl(1+\frac{1}{\sqrt\tau}\Bigr)\frac{i!}{\tau^i}
\ees
where $C$ is some positive constant. Finally
\bes
\sum_{i\ge0}\abs{y_{j,i}}\frac{x_2^{2i}}{(2i)!}
\leq C\Bigl(1+\frac{1}{\sqrt\tau}\Bigr)\sum_{i\ge0}\frac{i!}{(2i)!}\Bigl(\frac{x_2^2}{\tau}\Bigr)^i<\infty,
\ees
where the series in the rhs is obviously convergent.
\end{proof}
\begin{thm}\label{prop:nullcont}
Consider any $\theta _0\in L^2(\Omega )$, $T>0$, $\tau\in(0,T)$ and $s\in(1,2)$. Then there exists a sequence $(y_j)_{j\ge0}$ of Gevrey functions of order $s$ on~$[\tau,T]$ such that the control
\begin{IEEEeqnarray*}{rCl}
    u(t,x_1) &:=& \begin{cases}
    0& \text{if $0\le t\le\tau$},\\
    \displaystyle\sum_{i,j\ge 0}e^{-\lambda_jt}e_j(x_1)\frac{y_j^{(i)}(t)}{(2i-1)!} &\text{if $\tau<t\le T$}
\end{cases}
\end{IEEEeqnarray*}
steers the system from the initial state $\theta_0$ at time~$0$ to the final state~$0$ at time~$T$.

Moreover $u$ is Gevrey of order~$s$ in~$t$ and $1/2$ in~$x_1$ on $[0,T]\times[0,L]$; the solution $\theta$ of~\eqref{eq:heat}-\eqref{eq:bcu} belongs to $C\bigl([0,T],L^2(\Omega)\bigr)$ and is Gevrey of order $s$ in~$t$, $1/2$ in~$x_1$ and $s/2$ in~$x_2$  on~$[\ep,T]\times\overline{\Omega}$ for all $\ep\in(0,T)$.
\end{thm}
\begin{proof}
By lemma~\ref{lem:zerocontrol} the state reached at time~$\tau$ is
\begin{IEEEeqnarray*}{rCl}
    \theta_\tau(x_1,x_2) &=&
    \sum_{j\ge0}e^{-\lambda_j\tau}e_j(x_1)\sum_{i\ge0}y_{j,i}\frac{x_2^{2i}}{(2i)!},
\end{IEEEeqnarray*}
with
\begin{IEEEeqnarray*}{rCl}
    y_{j,i} &=& \sqrt{2}\sum_{n\ge0}c_{j,n}e^{-n^2\pi^2\tau}(-n^2\pi^2)^i.
\end{IEEEeqnarray*}
Clearly the functions
\begin{IEEEeqnarray*}{rCl}
    \bar y_j(t) &:=& \sqrt{2}\sum_{n\ge0}c_{j,n}e^{-n^2\pi^2t}
\end{IEEEeqnarray*}
satisfy $\overline y_j^{(i)}(\tau)=y_{j,i}$ for all $i,j\geq0$. Hence the desired functions $y_j$ are given by
\begin{IEEEeqnarray*}{rCl}
    y_j(t) &:=& \phi_s\Bigl(\frac{t-\tau}{T-\tau}\Bigr)\,\overline y_j(t)
\end{IEEEeqnarray*}
($\phi_s$ is defined at the end of the introduction). Indeed by lemma~\ref{lem:zerocontrol} the growth property of the~$y_{j,i}$'s implies the $\overline{y}_j$'s are analytic on~$[\tau,T]$ hence Gevrey of order~$s>1$; since $\phi(t):=\phi_s\Bigl(\frac{t-\tau}{T-\tau}\Bigr)$ is also Gevrey of order~$s$ on~$[\tau,T]$, so is the product~$y_j$ of $\overline y_j$ by~$\phi$.

On the other hand the definition of~$\phi_s$ implies the $y_j$'s satisfy for all $i,j\geq0$
\begin{IEEEeqnarray*}{rCl}
y_j^{(i)}(\tau)  &=& y_{j,i}\\
y_j^{(i)}(T)  &=& 0,
\end{IEEEeqnarray*}
so that $\theta(\tau^+,x_1,x_2)=\theta_\tau(x_1,x_2)$.

The proposed control~$u$ achieves the steering to zero: the solution $\theta$ of~\eqref{eq:heat}--\eqref{eq:bcu} is then given by~\eqref{eq:theta} and obviously satisfies~$\theta(T,x_1,x_2)=0$; by proposition~\ref{prop:regularity} it is Gevrey of order~$s$ in~$t$, $1/2$ in~$x_1$, and~$s/2$ in~$x_2$ on~$[\tau,T]\times\overline{\Omega}$. Moreover $u$ is also Gevrey of order~$s$ in~$t$ and $1/2$ in~$x_1$ on~$[0,T]\times[0,L]$.

It can be shown that for $\ep\leq t\leq\tau$ the solution $\theta$ of~\eqref{eq:heat}--\eqref{eq:bcu} is Gevrey of order~$1$ in~$t$, $1/2$ in~$x_1$, and~$s/2$ in~$x_2$ on~$[\tau,T]\times\overline{\Omega}$. The proof is omitted since this fact is not used to design the control. To complete the proof we then only need to check the compatibility at~$t=\tau$, ie
\begin{IEEEeqnarray*}{rCl}
\partial_t^k\theta(\tau^+,x_1,x_2) &=& \partial_t^k\theta(\tau,x_1,x_2)
\end{IEEEeqnarray*}
for all $k\ge0$ and $x\in\overline{\Omega}$. Indeed using the notations in the proof of lemma~\ref{lem:zerocontrol}
\begin{IEEEeqnarray*}{rCl}
\IEEEeqnarraymulticol{3}{l}{
\partial_t^k\theta(\tau^+,x_1,x_2)
}\\ ~~
&=& \sum_{i,j\ge0}e_j(x_1)\frac{x_2^{2i}}{(2i)!}\,
\partial_t^k\bigl[e^{-\lambda_jt}y_j^{(i)}(t)\bigr]\Bigl\vert_{t=\tau}\\
&=& \sum_{i,j\ge0}e_j(x_1)\frac{x_2^{2i}}{(2i)!}\,  \sum_{m=0}^k  \binom{k}{m} (-\lambda_j)^m e^{-\lambda_j\tau}  y_{j,i+k-m}  \\
&=& \sum_{i,j\ge0}e_j(x_1)\frac{x_2^{2i}}{(2i)!}\,  \sum_{m=0}^k  \binom{k}{m} (-\lambda_j)^m e^{-\lambda_j\tau}\\
&& \quad\times\sqrt2\sum_{n\ge0}c_{j,n}e^{-n^2\pi^2\tau}(-n^2\pi^2)^{i+k-m}\\
&=& \sqrt2\sum_{j,n\ge0}e_j(x_1)c_{j,n}e^{-(\lambda_j+n^2\pi^2)\tau}\,
\sum_{i\ge0}(-1)^i\frac{(n\pi x_2)^{2i}}{(2i)!}\\
&& \quad\times \sum_{m=0}^k  \binom{k}{m} (-\lambda_j)^m(-n^2\pi^2)^{k-m}\\
&=& \sum_{j,n\ge0}(-\lambda_j-n^2\pi^2)^k c_{j,n}e^{-(\lambda_j+n^2\pi^2)\tau}\\
&& \quad\times e_j(x_1)
\sqrt{2}\cos(n\pi x_2)\\
&=& \partial_t^k\theta(\tau,x_1,x_2).
\end{IEEEeqnarray*}

Hence the solution $\theta$ of~\eqref{eq:heat}-\eqref{eq:bcu} belongs to $C\bigl([0,T],L^2(\Omega)\bigr)$ and is Gevrey of order $s$ in~$t$, $1/2$ in~$x_1$ and $s/2$ in~$x_2$  on~$[\ep,T]\times\overline{\Omega}$ for all $\ep\in(0,T)$. As a consequence $u$ is Gevrey of order~$s$ in~$t$ and $1/2$ in~$x_1$ on $[0,T]\times[0,L]$.
\end{proof}

\section{Numerical experiments}\label{sec:numerics}
We illustrate the approach on a numerical example. The parameters are
$L=1$, $T=0.3$, $\tau=0.05$, $s=1.65$. The initial condition is the ``double step''
\begin{IEEEeqnarray*}{rCl}
    \theta_0(x_1,x_2) &:=& \begin{cases}
    -1& \text{if $(x_1,x_2)\in(0,\frac{1}{2})\times(0,\frac{1}{2})$},\\
    1& \text{if $(x_1,x_2)\in(0,\frac{1}{2})\times(\frac{1}{2},1)$},\\
    1& \text{if $(x_1,x_2)\in(\frac{1}{2},1)\times(0,\frac{1}{2})$},\\
    -1& \text{if $(x_1,x_2)\in(\frac{1}{2},1)\times\times(\frac{1}{2},1)$}.
\end{cases}
\end{IEEEeqnarray*}
Its nonzero Fourier coefficients are
\begin{IEEEeqnarray*}{rCl'l}
    c_{2l+1,2p+1} &=& -2\frac{(-1)^{l+p}}{\pi^2(2l+1)(2p+1)}, & l,p\geq0.
\end{IEEEeqnarray*}
Notice $\theta_0$ is not continuous and that its Fourier coefficients decay fairly slowly. The series~\eqref{eq:theta}-\eqref{eq:u} and Fourier expansions have been truncated at a ``large enough'' order for a good accuracy.

The following figures are snapshots of the time evolution of the temperature from $t=0$ to $t=T-0.025$ by increments of~$0.025$  (the snapshot at $t=T$, being exactly~$0$, is omitted), followed by the complete time evolution of the control effort.
An interesting question to study in the future is the tradeoff between $\tau$ and~$T-\tau$ to reach the zero state with the smallest control effort: longer regularization~$\tau$ or longer duration~$T-\tau$ of the active control?

\centering
\includegraphics[width=0.61\columnwidth]{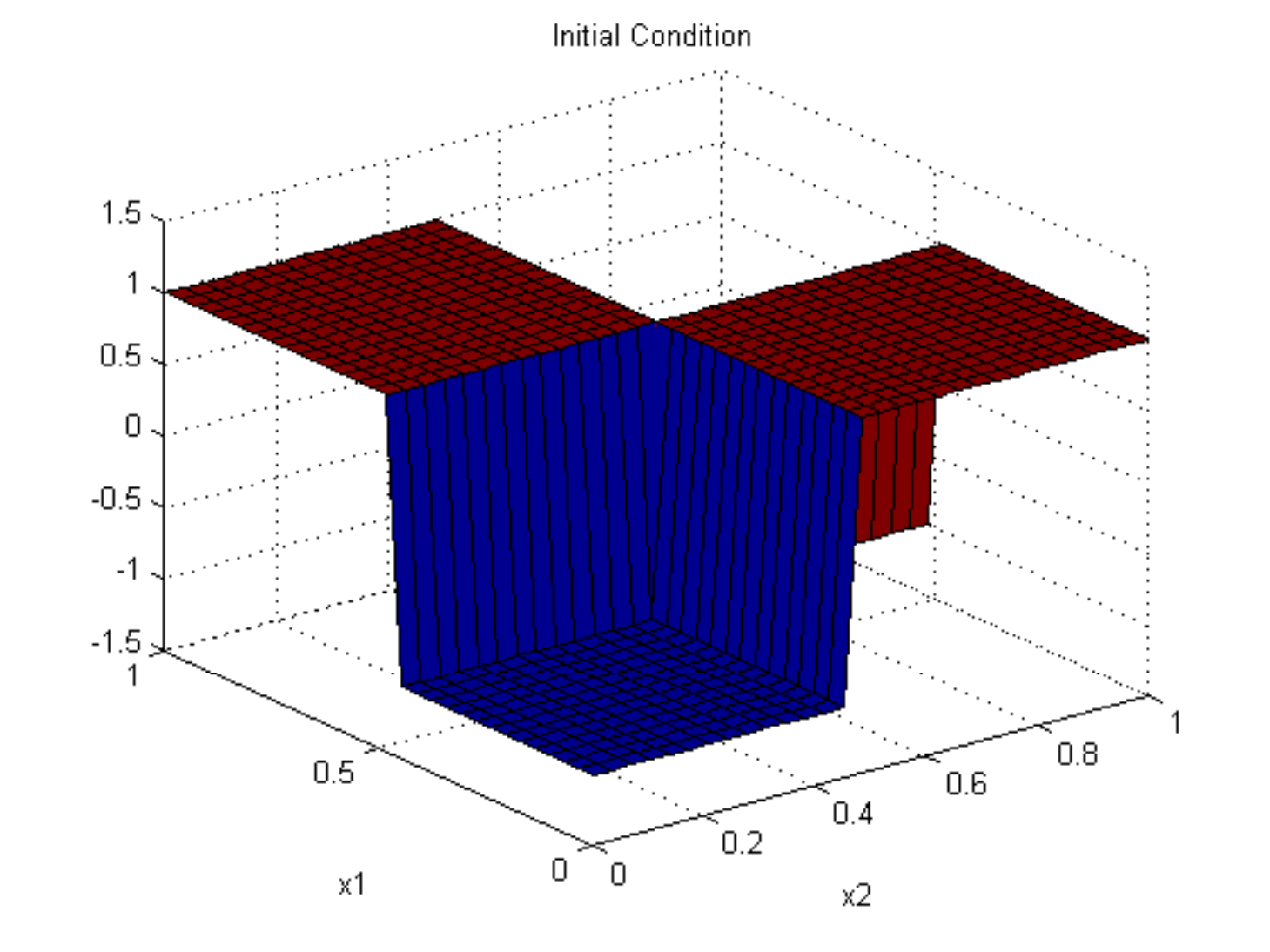}
\includegraphics[width=0.61\columnwidth]{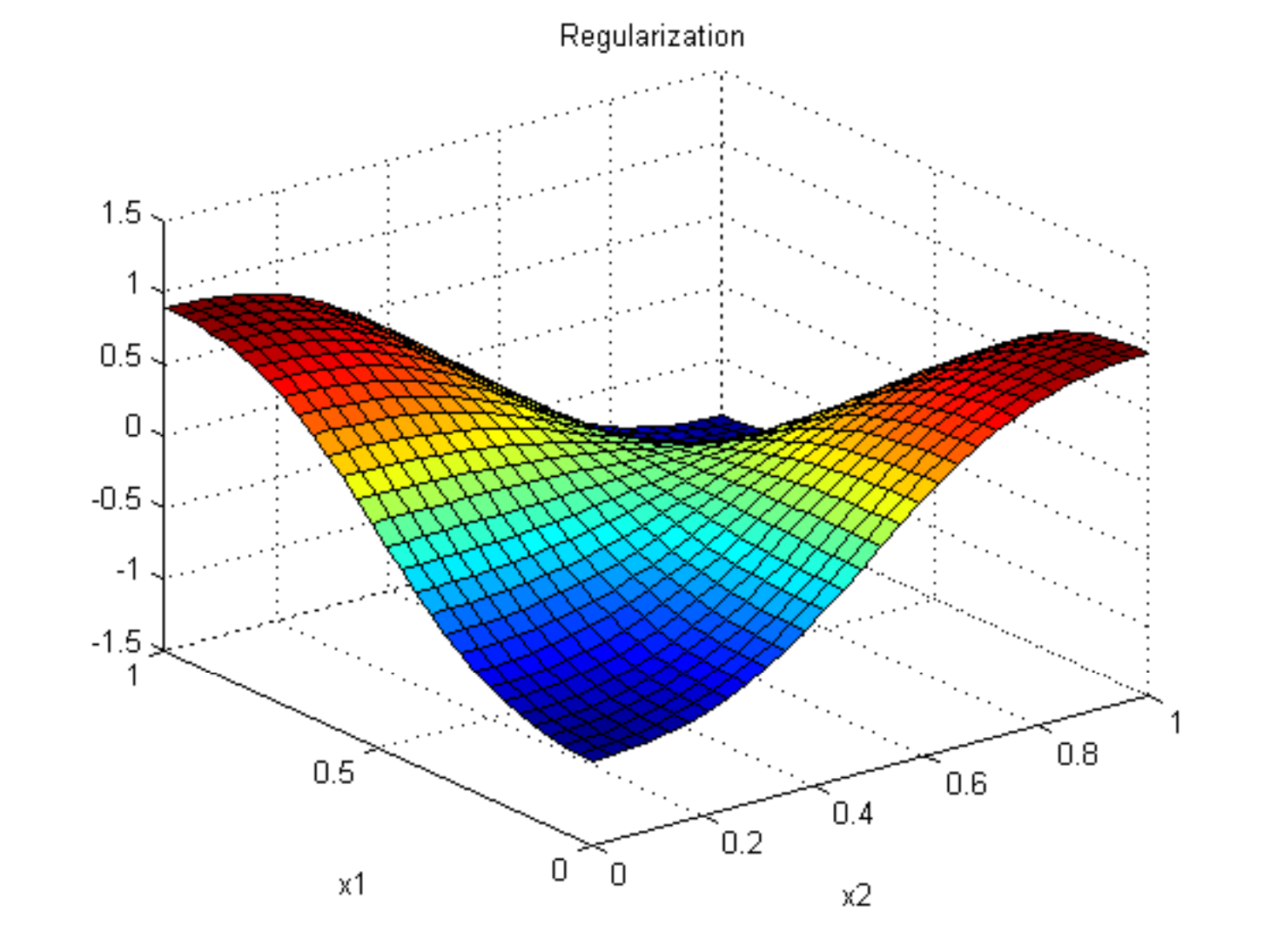}
\includegraphics[width=0.61\columnwidth]{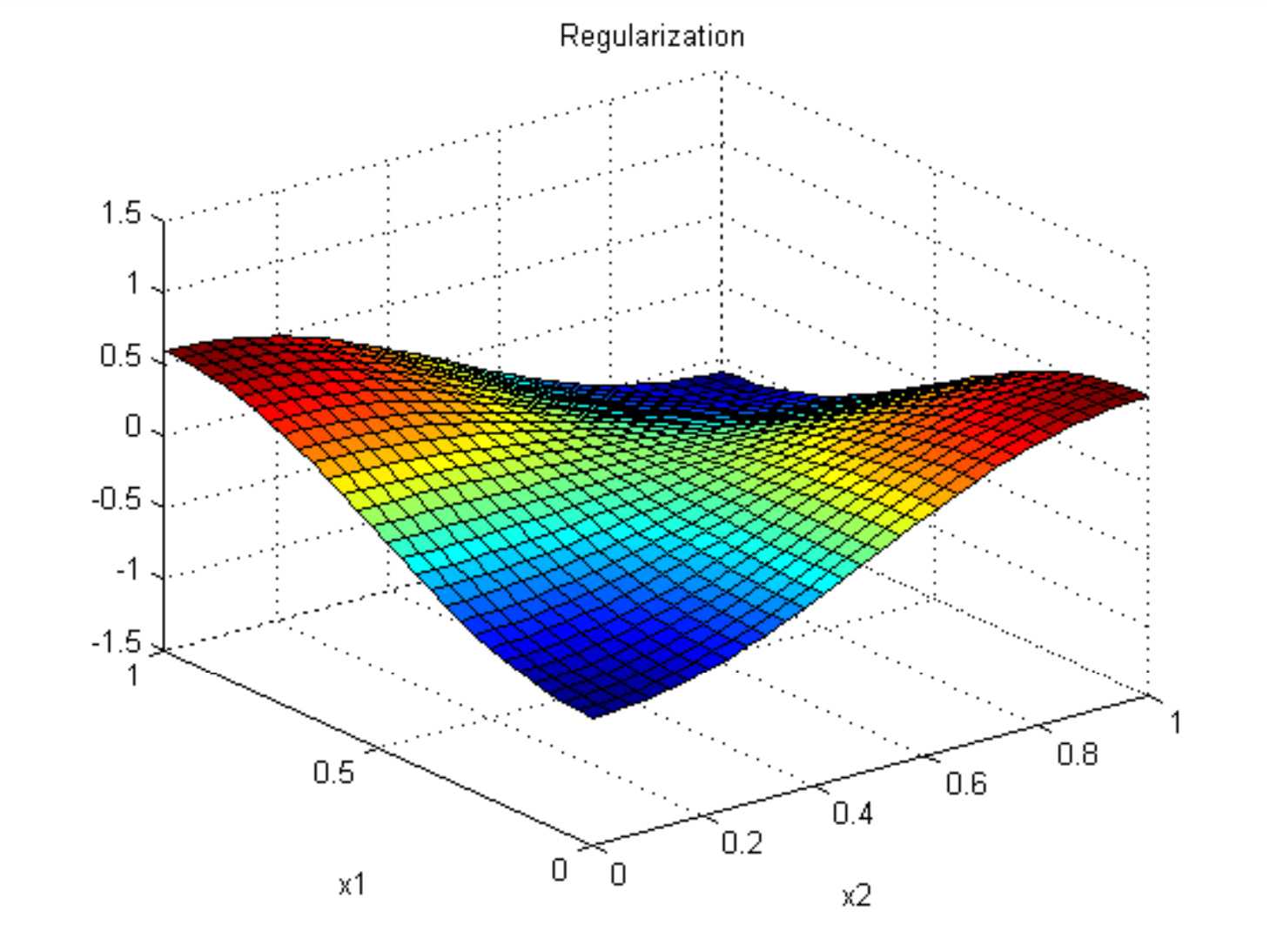}
\includegraphics[width=0.61\columnwidth]{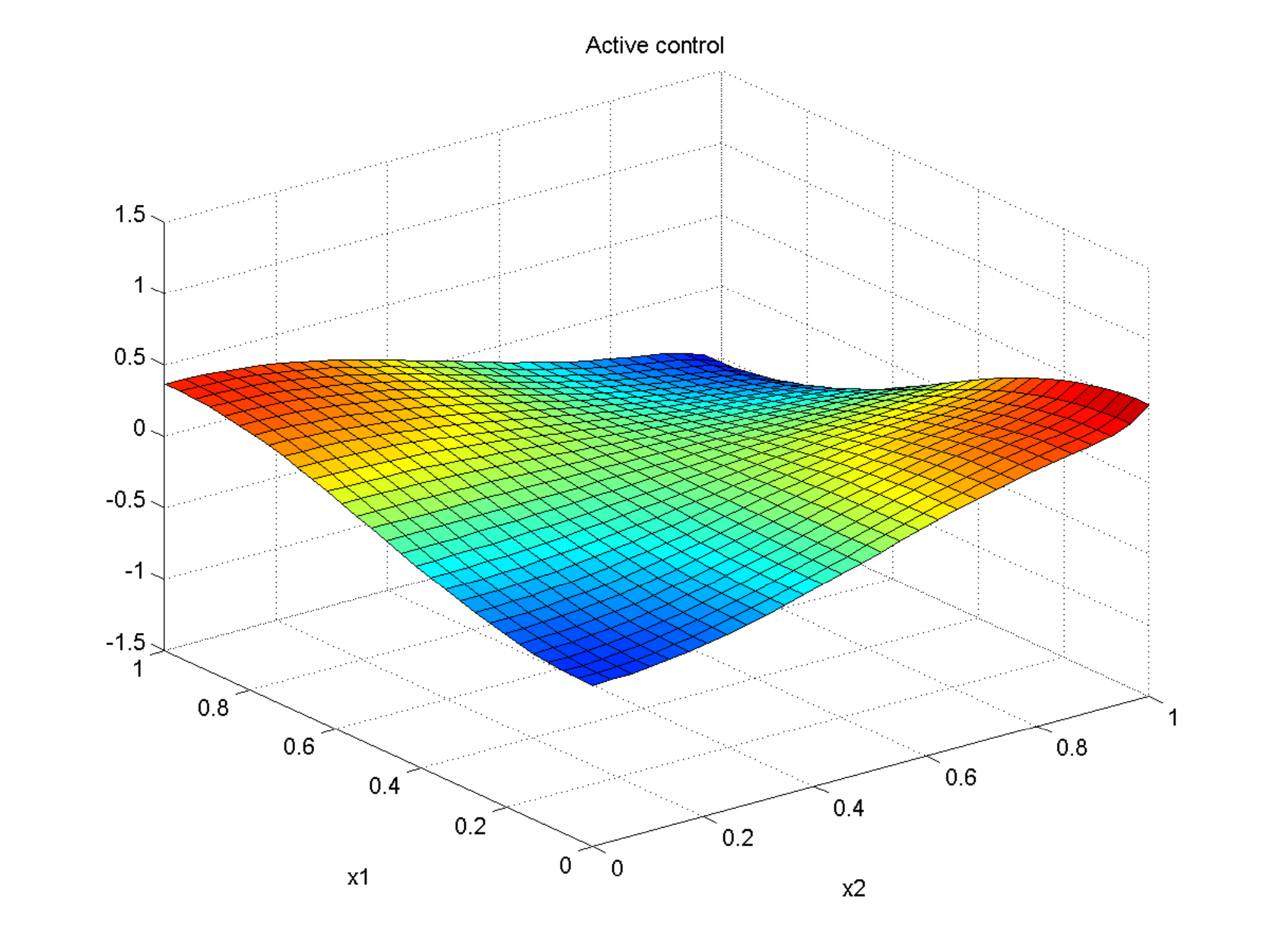}
\includegraphics[width=0.61\columnwidth]{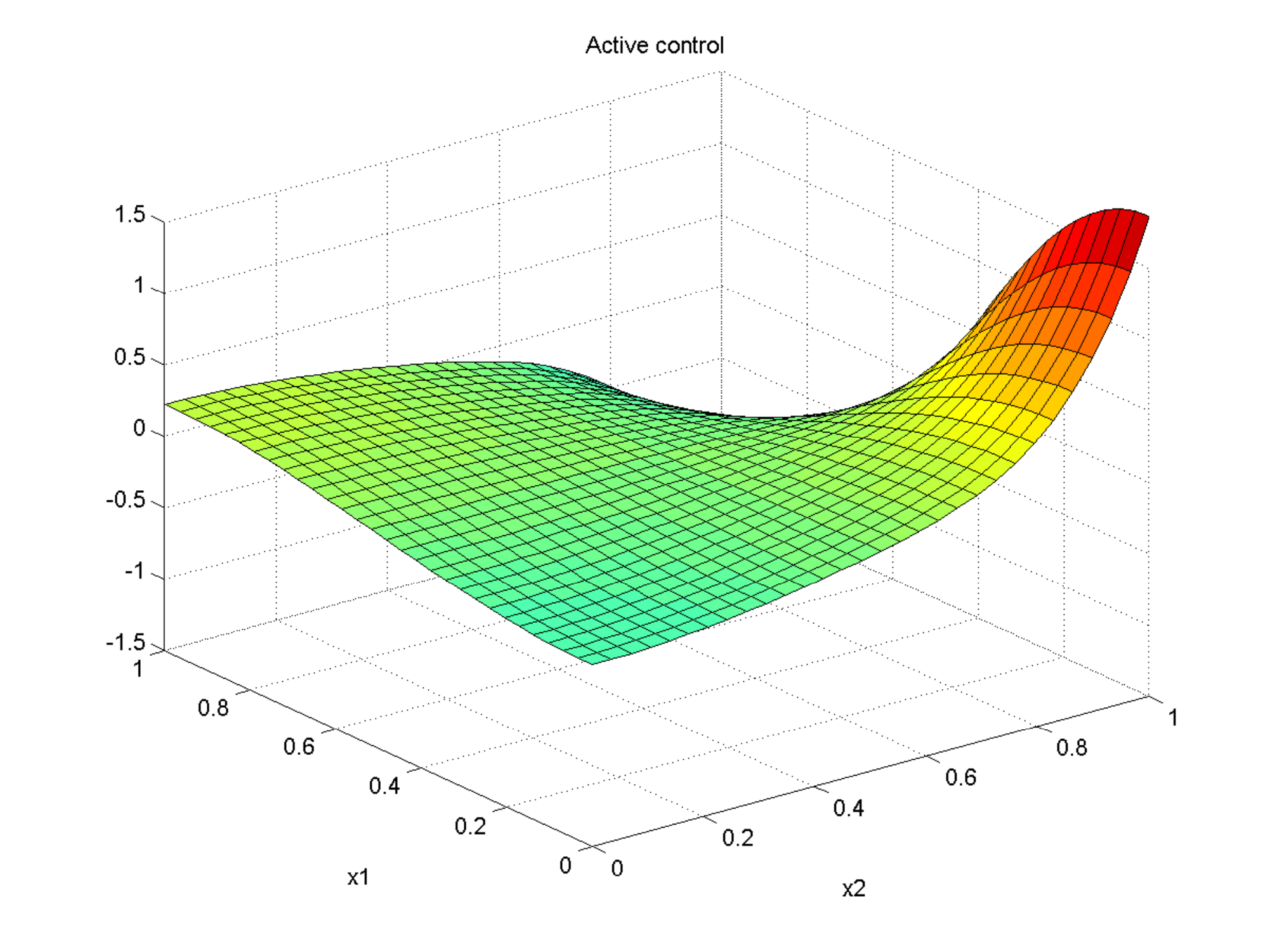}
\includegraphics[width=0.61\columnwidth]{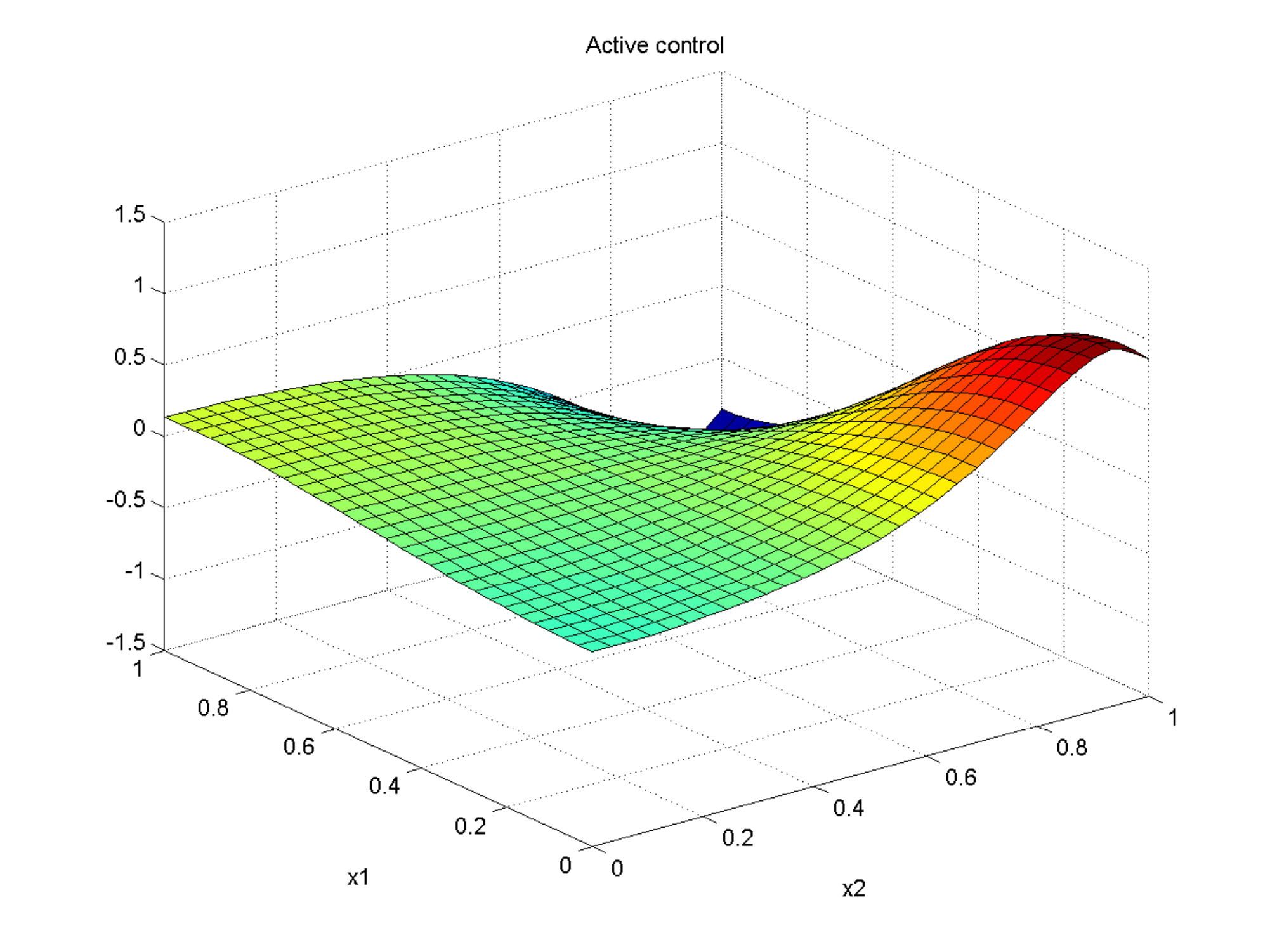}
\includegraphics[width=0.61\columnwidth]{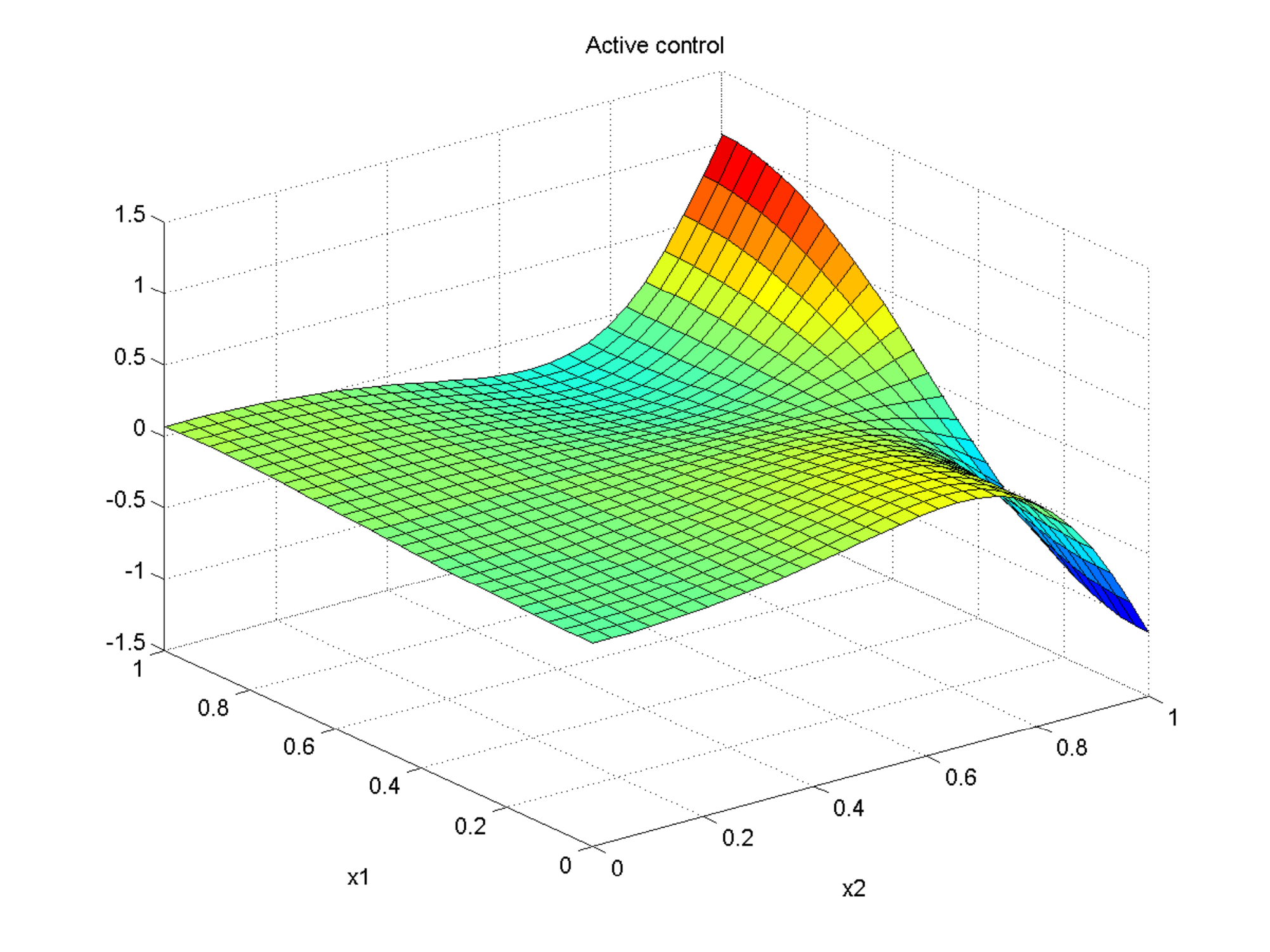}
\includegraphics[width=0.61\columnwidth]{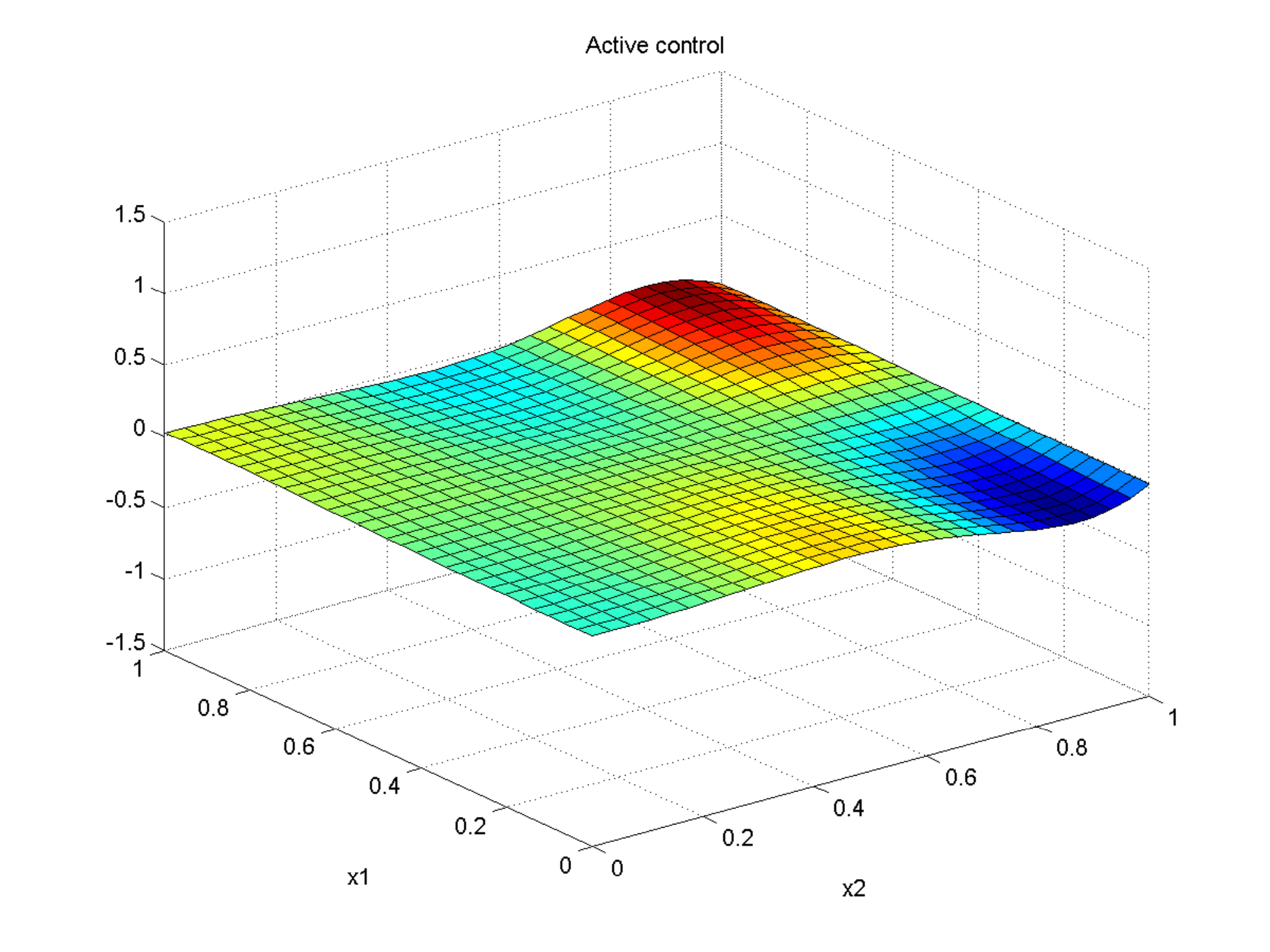}
\includegraphics[width=0.61\columnwidth]{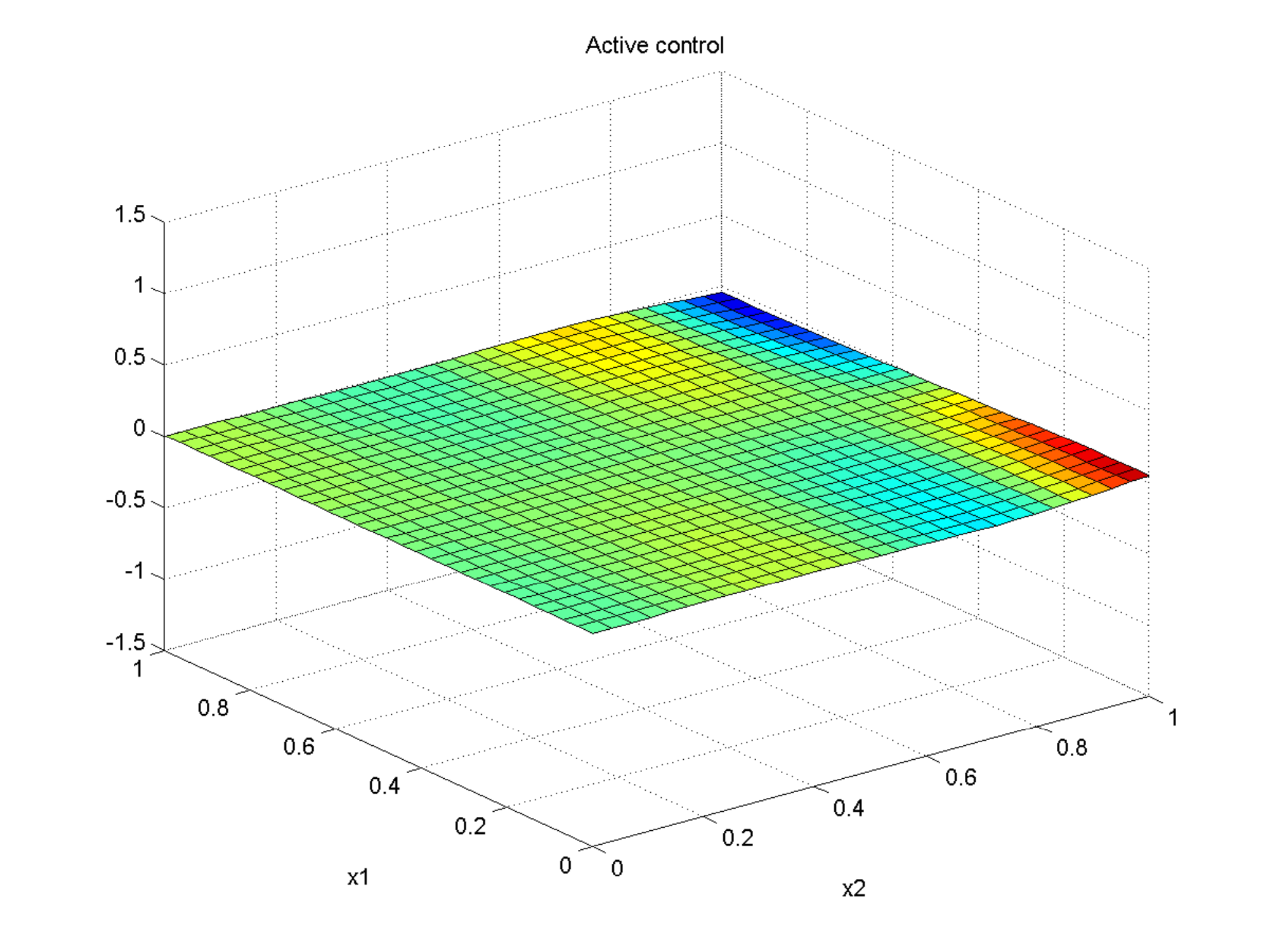}
\includegraphics[width=0.61\columnwidth]{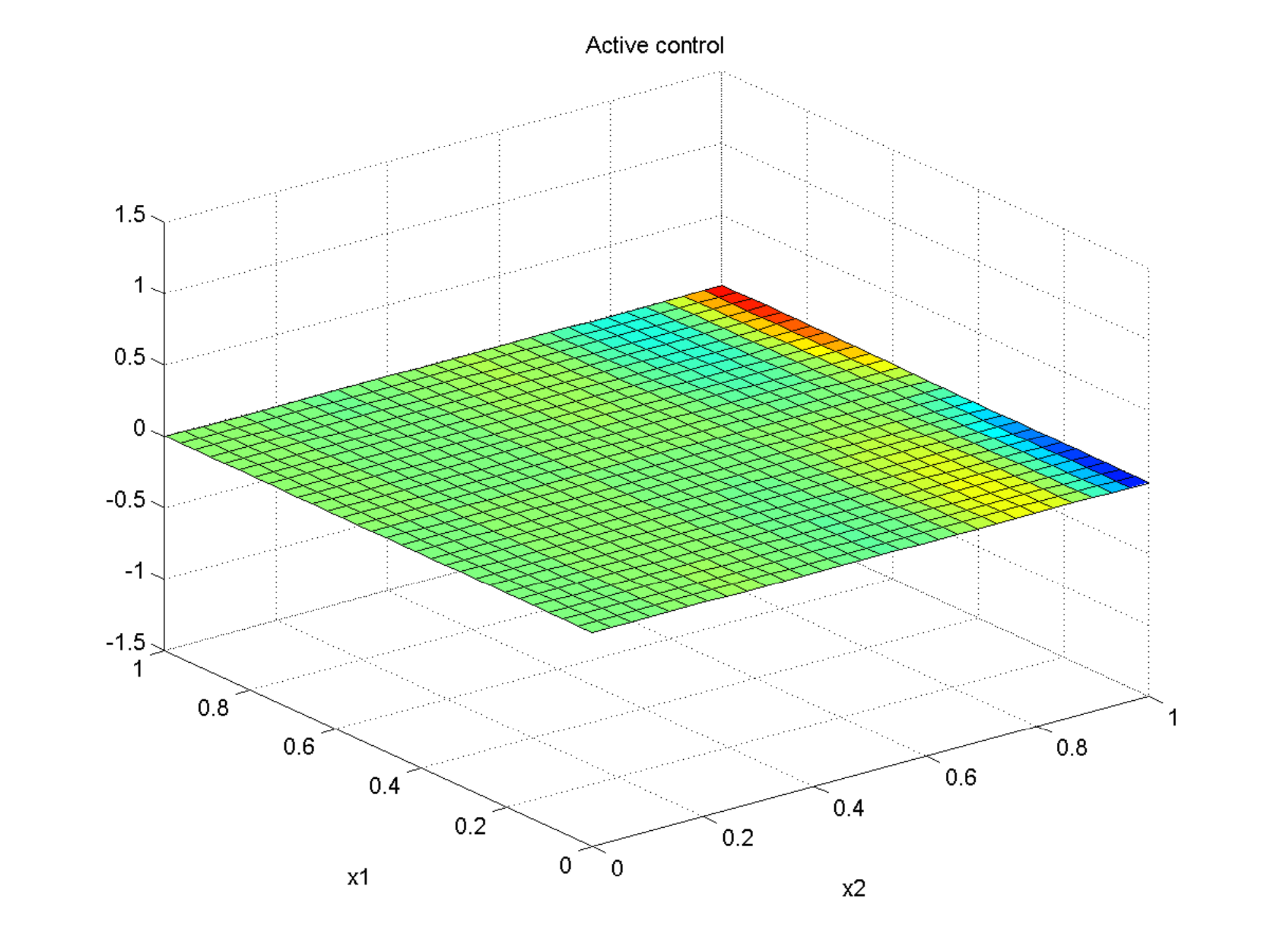}
\includegraphics[width=0.61\columnwidth]{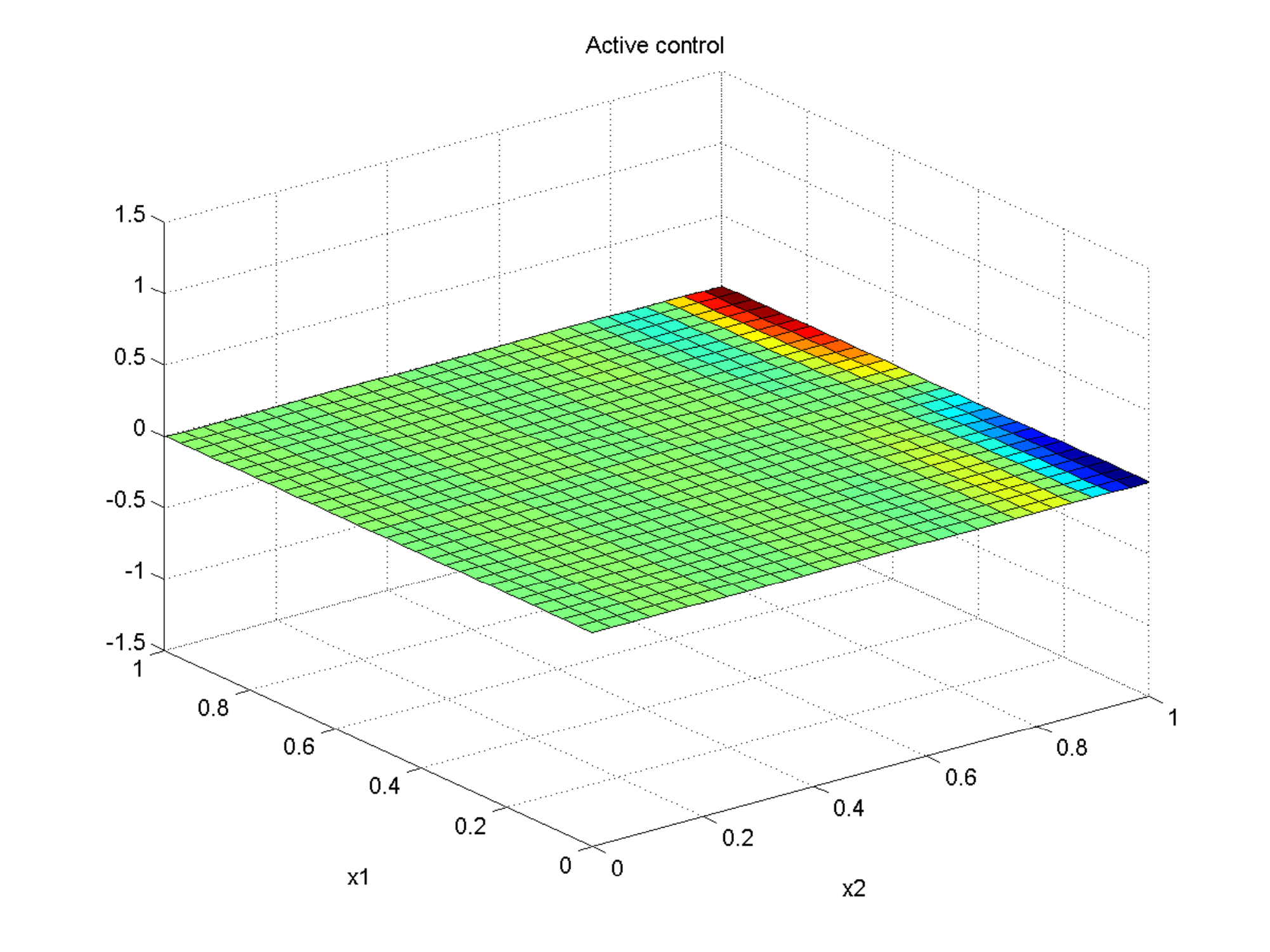}
\includegraphics[width=0.61\columnwidth]{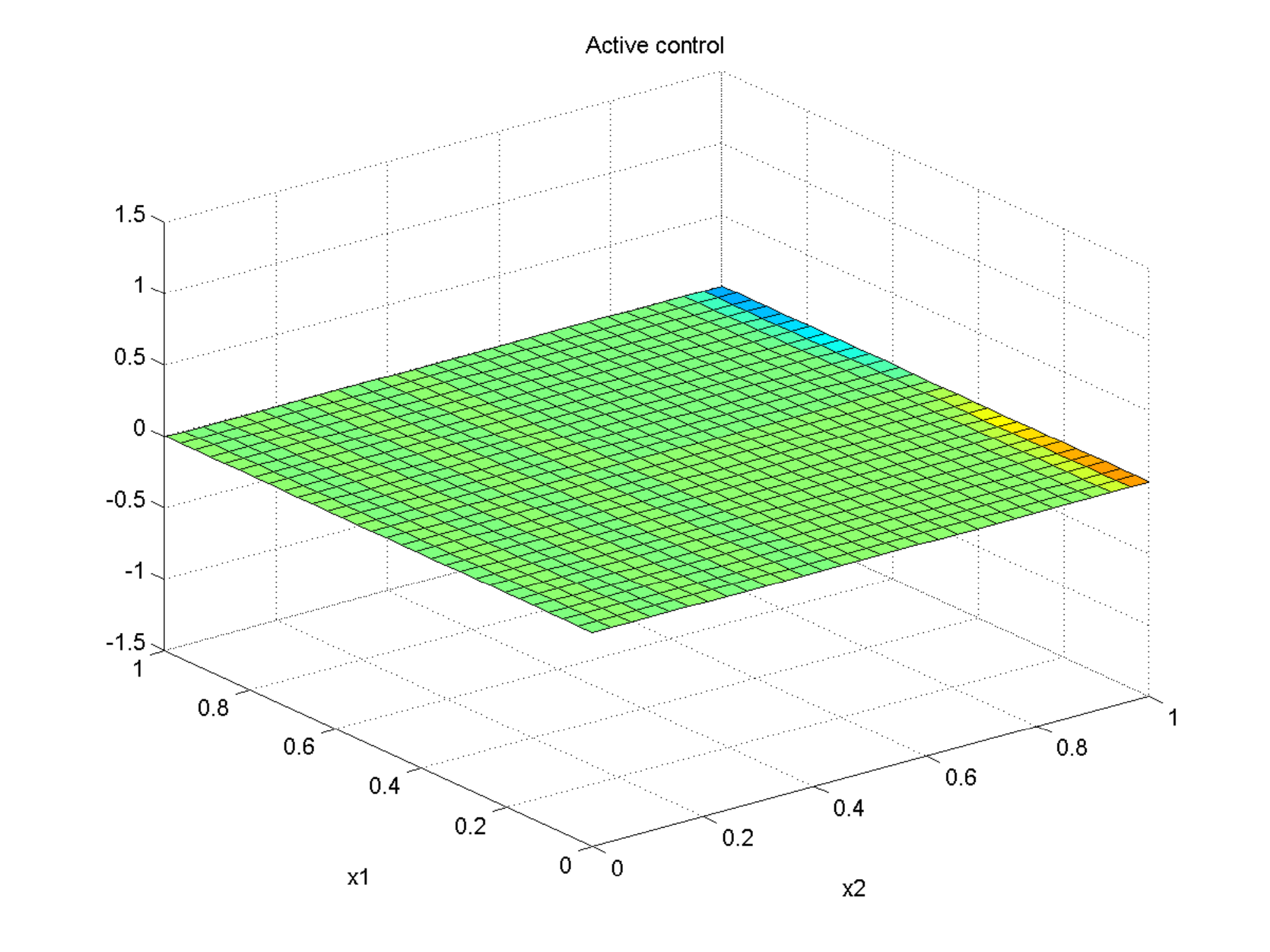}
\includegraphics[width=0.95\columnwidth]{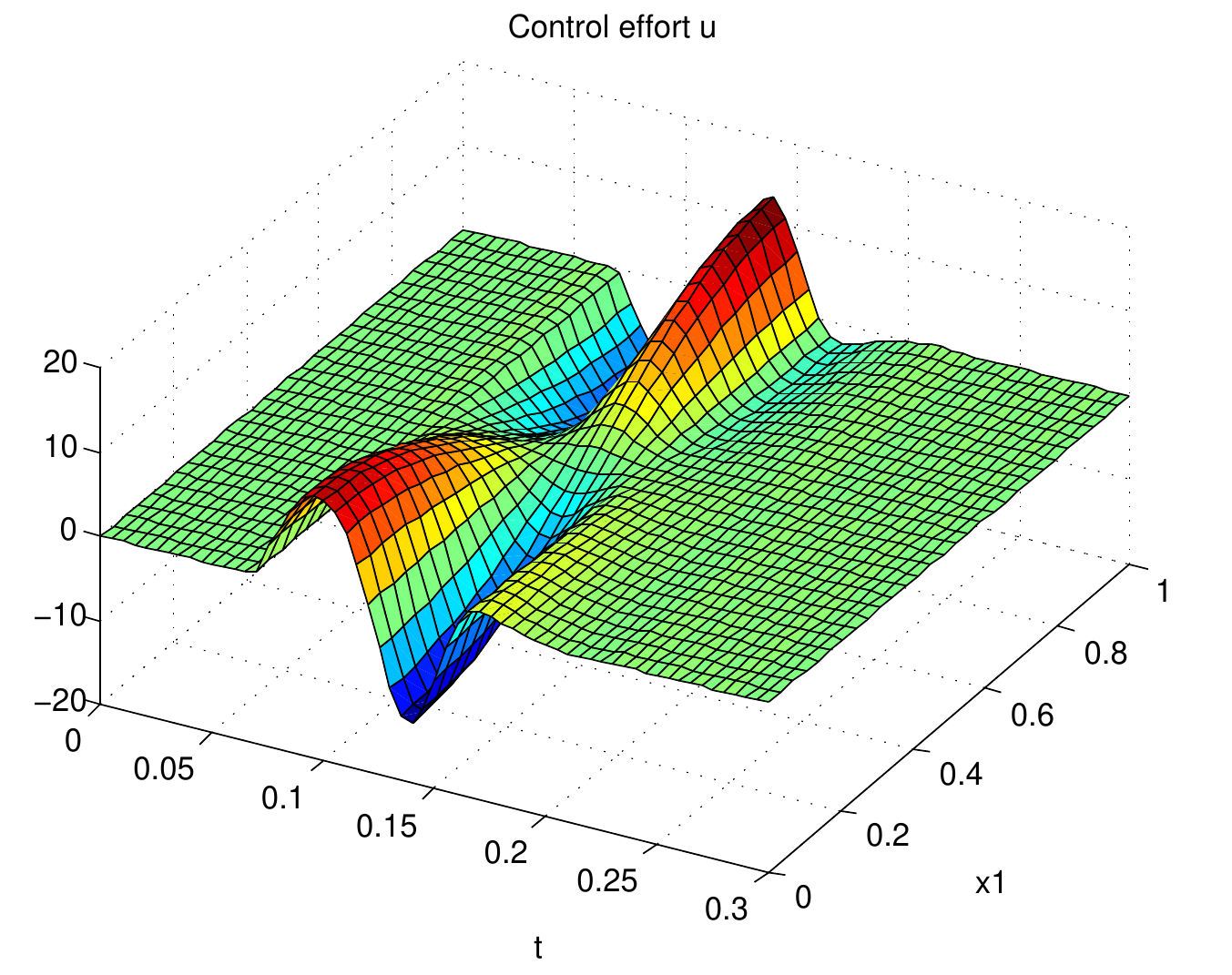}




\bibliographystyle{phmIEEEtran}
\bibliography{cdc}

%
%
%
%

\end{document}